\documentclass[12pt]{article}
\usepackage{booktabs}
\usepackage{caption}
\usepackage{mathrsfs}
\usepackage{amsmath}
\usepackage{amsfonts,amsthm,amssymb,mathrsfs,bbding}
\usepackage{graphics,multicol}
\usepackage{graphicx}
\usepackage{color}
\usepackage{enumerate}
\usepackage{caption}
\usepackage{rotating}
\usepackage{lscape}
\usepackage{longtable}
\allowdisplaybreaks[4]
\usepackage{tabularx}
\usepackage[colorlinks=true,anchorcolor=blue,filecolor=blue,linkcolor=blue,urlcolor=blue,citecolor=blue]{hyperref}
\usepackage{extarrows}
\usepackage{cite}
\usepackage{latexsym,bm}
\usepackage{mathtools}
\evensidemargin=\oddsidemargin\topmargin=-1.5cm

\newtheorem{thm}{Theorem}

\newtheorem{lem}{Lemma}

\newtheorem{claim}{Claim}
\theoremstyle{definition}
\renewcommand\proofname{\it Proof}
\addtocounter{section}{0}
\begin{document}

	\title{\bf Spectral conditions for graphs having all (fractional) $[a,b]$-factors}
	\author{{Jiaxin Zheng, Junjie Wang, and  Xueyi Huang\footnote{Corresponding author.}\setcounter{footnote}{-1}\footnote{\emph{E-mail address:} huangxymath@163.com}}\\[2mm]
		\small School of Mathematics, East China University of Science and Technology,\\
		\small Shanghai 200237, China}

	\date{}
	\maketitle
	{\flushleft\large\bf Abstract }  Let $a\leq b$ be two positive integers. We say that a graph $G$ has all $[a,b]$-factors if it has an $h$-factor for every function $h: V(G)\rightarrow \mathbb{Z}^+$ such that $a\le h(v) \le b$ for all $v\in V(G)$ and $\sum_{v\in V(G)}h(v)\equiv 0\pmod 2$, and  has all fractional $[a,b]$-factors if it has a fractional $p$-factor for every $p: V(G) \rightarrow \mathbb{Z}^+$ such that $a\le p(v)\le b$ for all $v\in V(G)$. In this paper, we provide tight spectral radius conditions for graphs having all $[a,b]$-factors ($3\leq a<b$) and all fractional $[a,b]$-factors ($1\leq a<b$), respectively. 
	
	\begin{flushleft}
	\end{flushleft}

	\section{Introduction}
	All graphs considered in this paper are undirected and simple. Let $G$ be a graph with vertex set $V(G)$ and  edge set $E(G)$, and let $e(G)=|E(G)|$ denote the number of edges in $G$.  For any $v\in V(G)$, we denote by $d_G(v)$ be the degree of $v$ in $G$, $N_G(v)$ the set of vertices adjacent to $v$ in $G$, and $E_G(v)$ the set of edges incident with $v$ in $G$.  For any $S\subseteq V(G)$, we denote $E_G(S)=E(G[S])$ and $e_G(S)=|E_G(S)|$, where $G[S]$ is the subgraph of $G$ induced by $S$. Also, we denote by $E_G(S,T)$ the set of edges between two disjoint subsets $S$ and $T$ of $V(G)$, and write $e_G(S,T)=|E_G(S,T)|$.  The \textit{join} of two graphs $G_1$ and $G_2$, denoted by $G_1\nabla G_2$,  is the graph obtained from $G_1\cup G_2$ by adding all possible edges between $G_1$ and $G_2$.

%The \textit{adjacency matrix} of $G$ is defined as $A(G)=(a_{u,v})_{u,v\in V(G)}$, where $a_{u,v}=1$ if $u$ and $v$ are adjacent in $G$, and $a_{u,v}=0$ otherwise. The largest eigenvalue of  $A(G)$ is called the  \textit{spectral radius} of $G$, and denoted by $\rho(G)$. For some basic results on the spectral radius of graphs, we refer the reader to \cite{CDS,S}, and references therein. 

%Let us relate a non-negative integer $f(x)$ to each vertex $x$ of the graph. If a subgraph has a degree of $f(x)$ at each vertex $x$, it is called an \textit{$f$-factor}. So a natural question arises: What condition does a graph satisfy that contains an $f$-factor? 

In the past seventy years, the theory of graph factors played a key role in the study of graph theory \cite{AK11,C,KT,Lov,LZ,N,T,TT}.  Let $g$ and $f$ be two integer-valued functions defined on $V(G)$ such that $0 \le g(v) \le f(v)$ for all $v \in V (G)$. A \textit{$(g,f)$-factor} of $G$ is a spanning subgraph $F$ of $G$ satisfying $g(v) \le d_{F}(v) \le f(v)$ for all $v\in V(G)$. In particular, an $(f,f)$-factor is called an \textit{$f$-factor}. 
Let $a\leq b$ be two positive integers. If $g\equiv a$ and $f\equiv b$, then a $(g,f)$-factor is also called an $[a,b]$-\textit{factor}. In 1952, Tutte \cite{T} established the famous $f$-Factor Theorem, which provides a necessary and sufficient condition for the existence of an $f$-factor in a graph. In 1970, Lov\'{a}sz \cite{Lov} generalized  the conclusion of the $f$-Factor Theorem to $(g,f)$-factors.

\begin{thm}(Lov\'{a}sz\cite{Lov}\label{thm::1})
	A graph $G$ has a $(g, f)$-factor if and only if
$$
		f(D)-g(S)+\sum_{x\in S}d_{G-D}(x)-\hat{q}_{G}(D,S,g,f) \ge 0
$$
	for all disjoint sets $D, S \subseteq V$, where $\hat{q}_{G}(D,S,g,f)$ denotes the number of components $C$ of $G-(D \cup S)$ with $g(v)=f(v)$ for all $v \in V(C)$ and $e_{G}(V(C), S)+f(V(C)) \equiv 1 \pmod 2$.
\end{thm}

Based on Theorem \ref{thm::1}, a series of sufficient conditions have been obtained for the existence of a $(g,f)$-factor (or  particularly, $[a,b]$-factor) in  graphs \cite{A,EK96,FLL22,LM98,L,LL,O2,O3,O4,WZ,X}. We say that a graph $G$ has \textit{all $(g,f)$-factors} if it has an $h$-factor for every function $h: V(G)\rightarrow \mathbb{Z}^+$ such that $g(v)\le h(v) \le f(v)$ for all $v\in V(G)$ and $\sum_{v\in V(G)}h(v)\equiv 0\pmod 2$. In \cite{N}, Niessen provided a characterization for graphs having all $(g,f)$-factors.

\begin{thm}(Niessen \cite{N}\label{thm::2})
	$G$ has all $(g, f)$-factors if and only if
$$
		g(D)-f(S)+\sum_{x\in S}d_{G-D}(x)-q_{G}^{*}(D,S,g,f) \ge \left\{
		\begin{array}{ll}
			-1,&\mbox{if $f \neq g$}\\
			0,&\mbox{if f = g}
		\end{array}
		\right.
$$
	for all disjoint sets $D, S \subseteq V$, where $q_{G}^{*}(D,S,g,f)$ denotes the number of components $C$ of $G-(D \cup S)$ such that there exists a vertex $v \in V(C)$ with $g(v) < f(v)$ or $e_{G}(V(C), S)+f(V(C)) \equiv 1$ (mod $2$). 
\end{thm} 

In particular, we obtain the following characterization for graphs having all $[a,b]$-factors.

\begin{thm}\label{thm::5}
	Let $G$ be a graph and $a < b$ be two positive integers. Then $G$ has all $[a, b]$-factors if and only if for all disjoint  subsets $S, T \subseteq V(G)$, we have
	\begin{equation}\label{equ::1}
		\delta(S,T)=a|S|-b|T|+\sum_{x \in T}d_{G-S}(x)-q(S,T) \ge -1,
	\end{equation}
	where $q(S,T)$ denotes the number of connected components of $G-(S\cup T)$. 
\end{thm}

An alternative approach to $(g,f)$-factor is provided by the concept of fractional $(g,f)$-factor. Let  $h: E(G) \rightarrow [0, 1]$ be a function defined on $E(G)$ satisfying $g(v) \le \sum_{e \in E_G(v)}h(e) \le f(v)$ for all $v \in V(G)$. Setting $F_{h} = \{e: e \in E(G), h(e) > 0\}$. Then the subgraph of $G$ with vertex set $V(G)$ and edge set $F_h$, denoted by $G[F_h]$, is called a \textit{fractional $(g, f)$-factor} of $G$ with indicator function $h$. In particular, a fractional $(f,f)$-factor is called a \textit{fractional $f$-factor}, and if $g\equiv a$ and $f\equiv b$ with  $a\leq b$ being positive integers, then a fractional $(g,f)$-factor is called a \textit{fractional $[a,b]$-factor}. In \cite{A}, Anstee gave a necessary and sufficient condition for the existence of a fractional $(g,f)$-factor in a graph. Also, a new proof of Anstee's result was given by Liu and Zhang \cite{LZ}.

\begin{thm}(Anstee \cite{A}; Liu and Zhang \cite{LZ}\label{thm::3})
	Let $G$ be a graph and $g$, $f$: $V(G) \rightarrow \mathbb{Z}^+$ be two integer functions such that $g(v) \le f(v)$ for all $v \in V(G)$. Then $G$ has a fractional $(g, f)$-factor if and only if for any subset $S \subseteq V(G)$, we have
$$
		f(S)-g(T)+\sum_{v \in T}d_{G-S}(v) \ge 0,
$$
	where $T = \{ v\mid v \in V(G)-S$ and $d_{G-S}(v) < g(v)\}$. 
\end{thm}
For other simplified sufficient conditions for the existence of a fractional $(g,f)$-factor (or particularly, fractional $[a,b]$-factor) in graphs, see  \cite{BOW,O1,PLZ,ZS,ZZ}. We say that $G$ has \textit{all fractional $(g,f)$-factors} if it has a fractional $p$-factor for every function $p: V(G)\rightarrow \mathbb{Z}^+$ such that $g(v) \le p(v) \le f(v)$ for all $v\in V(G)$.  In \cite{LH},  Lu provided  a characterization for graphs having all fractional $(g,f)$-factors.

\begin{thm}(Lu \cite{LH}\label{thm::4})
	Let $G$ be a graph and $g$, $f$: $V(G) \rightarrow \mathbb{Z}^+$ be two integer functions such that $g(x) \le f(x)$ for all $x \in V(G)$. Then $G$ has all fractional $(g, f)$-factors if and only if for any subset $S \subseteq V(G)$, we have
$$
		g(S)-f(T)+\sum_{x \in T}d_{G-S}(x) \ge 0,
$$
	where $T = \{ v\mid v \in V(G)-S$ and $d_{G-S}(v) < f(v) \}$. 
\end{thm}

As a corollary of Theorem \ref{thm::4}, Lu \cite{LH} also obtained a characterization for graphs with all fractional $[a,b]$-factors.  

\begin{thm}(Lu\cite{LH})\label{thm::6}
	Let $G$ be a graph and $a < b$ be two positive integers. Then $G$ has all fractional $[a, b]$-factors if and only if for any subset $S \subseteq V(G)$, we have
	\begin{equation}\label{equ::2}
		\theta(S,T)=a|S|-b|T|+\sum_{x \in T}d_{G-S}(x) \ge 0,
	\end{equation}
	where $T = \{ v\mid v \in V(G)-S$ and $d_{G-S}(v) < b \}$. 
\end{thm}

%As an application of  Theorem \ref{thm::5}, Lu \cite{LH} obtained a sufficient condition for graphs having all fractional $[a,b]$-factors.

%\begin{thm}(Lu\cite{LH})\label{thm::6}
%	Let $a < b$ be two positive integers. Let $G$ be a graph with %order $n \ge 2(a+b)(a+b-1)/a$ and minimum degree $\delta_{G} \ge %\frac{(a+b-1)^2+4b}{4a}$. If $|N_{G}(u) \cup N_{G}(v)| \ge %\frac{bn}{a+b}$ for any two nonadjacent vertices $u$ and $v$ in %$G$, then $G$ has all fractional $[a, b]$-factors.
%\end{thm}

%Zhou and Sun \cite{ZS} posed a new neighborhood union condition for the existence of all fractional $[a,b]$-factors in graphs.

%\begin{thm}(Zhou and Sun\cite{ZS})\label{thm::7}
%	Let $a, b, r$ be three integers with $1 \le a \le b$ and $r %\ge 2$. Let $G$ be a graph of order $n$ with $n > %\frac{(a+b)(r(a+b)-2)}{a}$. If
%	\begin{equation}
%		\delta_{G} \ge \frac{(r-1)b^2}{a}
%		\nonumber
%	\end{equation}
%	and
%	\begin{equation}
%		|N_{G}(x_{1}) \cup N_{G}(x_{2}) \cup \cdots \cup %N_{G}(x_{r})| \ge \frac{bn}{a+b}
%		\nonumber
%	\end{equation}
%	for any independent subset $\{ x_{1}, x_{2}, \cdots , x_{r}\}$ %in $G$, then $G$ admits all fractional $[a, b]$-factors.
%\end{thm}

It is natural to ask whether there are some spectral conditions for graphs having all $[a, b]$-factors or  all fractional $[a, b]$-factors. 
In this paper, by using Theorem \ref{thm::5} and Theorem \ref{thm::6}, we provide tight spectral radius conditions for a graph to have all $[a,b]$-factors and all fractional $[a, b]$-factors, respectively.  

The \textit{adjacency matrix} of  $G$ is defined as $A(G)=(a_{u,v})_{u,v\in V(G)}$, where $a_{u,v}=1$ if $u$ and $v$ are adjacent in $G$, and $a_{u,v}=0$ otherwise, and the \textit{spectral radius} of $G$, denoted by $\rho(G)$, is the largest eigenvalue of  $A(G)$.  For any two integers $n$ and $b$ with $2 \leq b\leq n-1$, we denote $H_{n,b}:=K_{b-1}\nabla (K_1\cup K_{n-b})$. The main results of this paper are as below.

\begin{thm}\label{thm::7}
	Let $3\le a < b$ be integers, and let $G$ be a graph of order $n \geq 2b^2+4b$. If $\rho (G) \geq \rho (H_{n,b})$, then $G$ has all $[a,b]$-factors unless $G\cong H_{n,b}$.
\end{thm}

\begin{thm}\label{thm::8}
	Let $1\le a < b$ be integers, and let $G$ be a graph of order $n \geq 3b(b+a+1)/a+7$. If $\rho (G) \geq \rho (H_{n,b})$, then $G$ has all fractional $[a,b]$-factors unless $G\cong H_{n,b}$.
\end{thm}

	\section{Preliminaries}

        Let $M$ be an $n\times n$ matrix with real entries, and let $\Pi=\{X_{1},X_{2}, \ldots,X_{k}\}$ be a partition of $[n]=\{1,2,\ldots,n\}$. Then the matrix $M$ can be partitioned as
        $$
	M=\left(\begin{array}{ccccccc}
		M_{1,1}&M_{1,2}&\cdots&M_{1,k}\\
            M_{2,1}&M_{2,2}&\cdots&M_{2,k}\\
            \vdots&\vdots&\ddots&\vdots\\
            M_{k,1}&M_{k,2}&\cdots&M_{k,k}\\
	\end{array}\right).
	$$
 The \textit{quotient matrix} of $M$ with respect to $\Pi$ is  the $k\times k$ matrix $B_{\Pi}=(b_{i,j})^{k}_{i,j=1}$ with
        $$
            b_{i,j}=\frac{1}{|X_{i}|}\mathbf j^{T}_{|X_{i}|}M_{i,j}\mathbf j_{|X_{j}|}
        $$
for all $i,j \in \{ 1,2,\ldots,k \}$, where $\mathbf j_{s}$ denotes the all ones vector in $\mathbb{R}^{s}$. If each block $M_{i,j}$ of $M$ has constant row sum $b_{i,j}$, then $\Pi$ is called an \textit{equitable partition} of $M$, and correspondingly, $B_\Pi$ is called an  \textit{equitable quotient matrix} of $M$. As usual, if $M$ has only real eigenvalues, then we denote them by $\lambda_1(M)\geq \lambda_2(M)\geq \cdots \geq \lambda_n(M)$. 

        \begin{lem}(Brouwer and Haemers \cite[p. 30]{BH}; Godsil and Royle \cite[pp. 196--198]{GR})\label{lem::1}
            Let $M$ be a real symmetric matrix, and let $B$ be an equitable quotient matrix of $M$. Then the eigenvalues of $B$ are also eigenvalues of $M$. Furthermore, if $M$ is nonnegative and irreducible, then
            $$
                \lambda_1(M)=\lambda_1(B).
            $$
	\end{lem}
	
	\begin{lem}(Hong \cite{H})\label{lem::2}
		Let $G$ be a connected graph with $n$ vertices and $m$ edges. Then
		$$\rho(G)\le \sqrt{2m-n+1}.$$
	\end{lem}

    \begin{lem}\label{lem::3}
Let $a$, $b$ and $n$ be positive integers with  $b\geq a$ and $n \geq 3b(b+1)/a+3b+7$. Then
$$
\rho\left(K_{\left\lceil\frac{2b^2+2b}{a}\right\rceil+2b-4} \nabla \left(K_{2} \cup K_{n-\left\lceil\frac{2b^2+2b}{a}\right\rceil-2b+2}\right)\right)< n-2
$$
and
$$
\rho(K_{4b} \nabla (K_{2} \cup K_{n-4b-2}))<n-2.
$$
    \end{lem}

    \begin{proof}
Suppose $G_1=K_{\lceil(2b^2+2b)/a\rceil+2b-4} \nabla (K_{2} \cup K_{n-\lceil(2b^2+2b)/a\rceil-2b+2})$. Let $V_{1}=V(K_2)$, $V_2=V(K_{\lceil(2b^2+2b)/a\rceil+2b-4})$ and $V_3=V(K_{n-\lceil(2b^2+2b)/a\rceil-2b+2})$. It is easy to check that the partition  $\Pi: V(G_1)=V_1\cup V_2\cup V_3$ is an equitable partition of $A(G_1)$, and the corresponding  quotient matrix is 	
$$
	B_\Pi=
	\begin{pmatrix}
		1&\lceil\frac{2b^2+2b}{a}\rceil+2b-4&0\\
		2&\lceil\frac{2b^2+2b}{a}\rceil+2b-5&n-\lceil\frac{2b^2+2b}{a}\rceil-2b+2\\
		0&\lceil\frac{2b^2+2b}{a}\rceil+2b-4&n-\lceil\frac{2b^2+2b}{a}\rceil-2b+1\\
	\end{pmatrix}.
	$$
 Let $f(x)$ denote the characteristic polynomial of $B_\Pi$. As $n \geq 3b(b+1)/a+3b+7$, by a simple calculation, we obtain 
	$$
	\begin{aligned}
		f(n-2)=  n^2-4n-2\left(\left\lceil\frac{2b^2+2b}{a}\right\rceil\right)^2-(8b-14)\left\lceil\frac{2b^2+2b}{a}\right\rceil+28b-21> 0
	\end{aligned}
	$$
	and 
	$$f(n-3)=\det((n-3)I_n-B_\Pi)=-2\left(\left\lceil\frac{2b^2+2b}{a}\right\rceil+2b-4\right)^2<0.$$
We claim that $\lambda_1(B_\Pi)< n-2$. If not, then $f(n-2)>0$ and $f(n-3)<0$ imply that  $\lambda_2(B_\Pi)>n-2>\lambda_3(B_\Pi)>n-3$, and hence $\lambda_1(B_\Pi)+\lambda_2(B_\Pi)+\lambda_3(B_\Pi)> 3n-7$. On the other hand, $\lambda_1(B_\Pi)+\lambda_2(B_\Pi)+\lambda_3(B_\Pi)=\mathrm{trace}(B_\Pi)=n-3$, which is a contradiction. Therefore, by Lemma \ref{lem::1}, 
	$$
	\rho(G_1)=\lambda_1(B_\Pi)< n-2.
	$$
Similarly,  we also can prove that $\rho(K_{4b} \nabla (K_{2} \cup K_{n-4b-2}))< n-2$ because $n \geq 3b(b+1)/a+3b+7\geq \sqrt{32b^2-8b+1}+3$. This proves the lemma.
    \end{proof}

   \begin{lem}\label{lem::4}
   	Let $a$, $b$ and $n$ be positive integers.
   	\begin{enumerate}[(i)]
   		\item If $1\le a < b\leq n-1$, then  $H_{n,b}$  does not have all $[a, b]$-factors.
   		\item If $1\le a < b\leq n-2$, then  $H_{n,b}$  does not have all fractional $[a, b]$-factors.
   	\end{enumerate}
    \end{lem}

\begin{proof}
(i) Recall that $H_{n,b}=K_{b-1}\nabla (K_1\cup K_{n-b})$. Take $S=\emptyset$ and $T=V(K_1)$. Then we have
$$\delta(S,T)=a|S|-b|T|+\sum_{x\in T}d_{G-S}(x)-q(S,T)=-b+b-1-1=-2< -1,$$
contrary to \eqref{equ::1}. Therefore, by Theorem \ref{thm::5}, the graph $H_{n,b}$ cannot have all $[a, b]$-factors.

(ii) Take $S=\emptyset$ and $T=V(K_1)$. Since $n\geq b+2$, we see that $T$ contains all vertices of degree at most $b-1$ in $H_{n,b}-S=H_{n,b}$. Then we have
	$$\theta(S,T)=a|S|-b|T|+\sum_{x\in T}d_{G-S}(x)=-b+b-1=-1< 0,$$
contrary to \eqref{equ::2}. Therefore, by Theorem \ref{thm::6}, the graph  $H_{n,b}$ cannot have all fractional $[a, b]$-factors. 
\end{proof}

\section{Proof of the main results}

In this section, we give the proof of Theorems \ref{thm::7} and \ref{thm::8}.

{\flushleft \it Proof of Theorem \ref{thm::7}.}  Clearly, the graph $H_{n,b}=K_{b-1}\nabla (K_1\cup K_{n-b})$ is connected, and does not have all $[a, b]$-factors according to Lemma \ref{lem::4}. By assumption,  $\rho(G) \geq \rho(H_{n,b}) > \rho(K_{n-1})=n-2$ because $K_{n-1}$ is a proper subgraph of $H_{n,b}$. First notice  that  $G$ is connected. If not, let  $G_1, G_2,\ldots,G_r$ ($r\ge 2$) denote the connected components of $G$. As each $G_i$ is a subgraph of $K_{n-1}$, we have $\rho(G)=\max\{\rho(G_1),\rho(G_2),\ldots,\rho(G_r)\}\le \rho(K_{n-1})=n-2$, contrary to $\rho(G)>n-2$.  

Now suppose to the contrary that  $G$ does not have all $[a, b]$-factors and $G\ncong H_{n,b}$.  By Theorem \ref{thm::5}, there exist two disjoint vertex subsets $S, T\subseteq V(G)$ such that
\begin{equation}\label{equ::3}
	\delta(S,T)=a|S|-b|T|+\sum_{x \in T}d_{G-S}(x)-q(S,T) \le -2,
\end{equation}
where $q(S,T)$ denotes the number of connected components of $G-(S\cup T)$. We choose $S,T\subseteq V(G)$ satisfying \eqref{equ::3} such that $|S\cup T|$ is maximum. Let $s=|S|$, $t=|T|$ and $q=q(S,T)$. Suppose that  $C_1, C_2,\ldots, C_q$ are the connected components of $G-(S\cup T)$. We have the following two claims.

\begin{claim}\label{claim::1}
$e(\overline{G})\leq n-2$.
\end{claim}
\begin{proof}
Since $\rho(G)>n-2$, by Lemma \ref{lem::2}, we have $e(G)\geq \binom{n-1}{2}+1$, and hence $e(\overline{G})\leq n-2$. This proves Claim \ref{claim::1}.
\end{proof}

\begin{claim}\label{claim::2}
If $q\geq 1$, then $|V(C_i)|\geq 2$ for any $i\in \{1,2,\ldots,q\}$.
\end{claim}
\renewcommand\proofname{\it Proof}
\begin{proof}
Let $V_i=V(C_i)$. By contradiction, we assume that  $|V_i|=1$ for some $i\in \{1,2,\ldots,q\}$.  If $e_G(V_i, T)\leq b-1$, we let  $T'=T\cup V_i$. Then $q(S,T')=q-1$ and 
	$$
	\begin{aligned}
		\delta(S,T')&=a|S|-b|T'|+\sum_{x \in T'}d_{G-S}(x)-q(S,T')\\
		&= a|S|-b(|T|+1)+\left(\sum_{x \in T}d_{G-S}(x)+e_G(V_i, T)\right)-(q-1) \\
		&=\left(a|S|-b|T|+\sum_{x \in T}d_{G-S}(x)-q\right)-(b-1-e_G(V_i, T))  \\
		&\leq\delta(S,T)\\
		&\leq -2.
	\end{aligned}
	$$
Thus $S$ and $T'$ also satisfy \eqref{equ::3}, which contradicts the choice of $S$ and $T$. If $e_G(V_i, T)\geq b$, we let $S'=S\cup V_i$. Then $q_G(S',T)=q-1$ and
$$
\begin{aligned}
	\delta(S',T)&=a|S'|-b|T|+\sum_{x \in T}d_{G-S'}(x)-q''\\
	&= a(|S|+1)-b|T|+\left(\sum_{x \in T}d_{G-S}(x)-e_G(V_i, T)\right)-(q-1) \\
	&=\left(a|S|-b|T|+\sum_{x \in T}d_{G-S}(x)-q\right)+(a+1-e_G(V_i, T))  \\
	&\leq \delta(S,T)\\
	&\leq -2,
\end{aligned}
$$
contradicting the choice of $S$ and $T$.  This completes the proof of Claim \ref{claim::2}.
\end{proof}
For completing the proof, we consider the following three situations according to the value of $|T|$.

{\flushleft {\it Case 1.} $|T|=0$.}

If $|S|=0$, then $q \geq 2$  by (\ref{equ::3}), and hence $G=G-(S\cup T)$ is disconnected, which is impossible. Thus we can suppose that $|S|\geq 1$. By evaluating the number of edges among $C_1,\ldots,C_q$ in $\overline{G}$, we obtain  
$$
\begin{aligned}
	e(\overline{G})&\geq \frac{1}{2}\sum_{i=1}^{q}|V(C_i)|(n-|S|-|V(C_i)|)\\
	&\geq \sum_{i=1}^{q}(n-|S|-|V(C_i)|)~~\mbox{(as $|V(C_i)|\geq 2$ by Claim \ref{claim::2})}\\
	&=(q-1)(n-|S|).
\end{aligned}
$$
Combining this with $e(\overline{G})\leq n-2<n$ (by Claim \ref{claim::1}) yields that $n>(q-1)(n-|S|)$. Clearly, $q\leq (n-|S|)/2$  by Claim \ref{claim::2}. If $q\geq 3$, then $n<2|S|$ and $\delta(S,T)=a|S|-q\geq a|S|-(n-|S|)/2>a|S|-|S|/2>0$, contrary to \eqref{equ::3}. If $q\leq 2$, then $\delta(S,T)=a|S|-q\geq a|S|-2\geq -1$, again contrary to \eqref{equ::3}.

{\flushleft {\it Case 2.} $|T|=1$.}

In this situation, we suppose  $T=\{x_{0}\}$, and consider the following two cases.

{\flushleft {\it Subcase 2.1.} $|S|=0$.}

According to (\ref{equ::3}), we have $\delta(S,T)=-b+d_{G}(x_0)-q \le -2$. If $q\leq 1$, then $d_G(x_0)\leq b-1$, and hence $G$ is a spanning subgraph of $H_{n,b}$. As $G\ncong H_{n,b}$, we have $\rho(G)< \rho(H_{n,b})$, contrary to our assumption. If $q\geq 2$,  by Claim \ref{claim::2}, we see that  $G$ is a spanning subgraph of $K_1\nabla (K_{r}\cup K_{n-1-r})$ for some integer $r$ with $2\leq r\leq n-3$.  It is not difficult to verify that $\rho(K_1\nabla (K_{r}\cup K_{n-1-r}))< n-2$ whenever $2\leq r\leq n-3$. Thus, $\rho(G)<n-2$, a contradiction.

{\flushleft {\it Subcase 2.2.} $|S|\geq 1$.}

According to (\ref{equ::3}), we have $\delta(S,T)=a|S|-b+d_{G-S}(x_0)-q \le -2$. If $q\leq 1$, then  $d_{G-S}(x_{0}) \le b-a|S|-1$, which implies that  $d_G(x_{0}) \le b-a|S|-1+|S|=b+(1-a)|S|-1 \le b-1$. Thus $G$ is a spanning subgraph of $H_{n,b}$, which is impossible. Now suppose that $q\geq 2$. Since $e(\overline{G})\leq n-2$ by Claim \ref{claim::1}, there are at most $n-2-\frac{1}{2}\sum_{i=1}^{q}|V(C_i)|(n-|S|-1-|V(C_i)|)$ edges not in $E_{G}(T,V(G)\setminus(S\cup T))$. Therefore,  
$$
\begin{aligned}
	d_{G-S}(x_0)&\geq n-1-|S|-\left (n-2-\frac{1}{2}\sum_{i=1}^{q}|V(C_i)|(n-|S|-1-|V(C_i)|)\right)\\
	&=1-|S|+\frac{1}{2}\sum_{i=1}^{q}|V(C_i)|(n-|S|-1-|V(C_i)|)\\
	&\geq1-|S|+\sum_{i=1}^{q}(n-|S|-1-|V(C_i)|)~~\mbox{(as $|V(C_i)|\geq 2$)}\\
	&= q(n-|S|-1)-n+2.
\end{aligned}
$$
Recall that $n \geq 2b^2+4b$ and $a\geq 3$. We have
$$
\begin{aligned}
	\delta(S,T)&=a|S|-b+d_{G-S}(x_0)-q\\
	&\geq a|S|-b+q(n-|S|-1)-n+2-q  \\
	&\geq a|S|-b+2(n-|S|-2)-n+2  \\
	&=(a-2)|S|+n-b-2  \\
	&\geq (2b^2+4b)-b-2\\
	&=2b^2+3b-2\\
	&\geq 0,
\end{aligned}
$$
which is contrary to (\ref{equ::3}).

{\flushleft {\it Case 3.} $|T|\geq 2$.}

{\flushleft {\it Subcase 3.1.} $|S|=0$.}

In this situation, from (\ref{equ::3}) we obtain $\delta(S,T)=-b|T|+\sum_{x \in T}d_{G}(x)-q \le -2$. If $q\leq 1$, then  $\sum_{x \in T}d_{G}(x) \le b|T|-1$, and hence there exists some vertex $x_1\in T$ such that $d_G(x_1)\leq b-1$. Therefore, $G$ is a spanning subgraph of $H_{n,b}$, which is impossible. Now suppose that $q\geq 2$.  As in Subcase 2.2, there are at most $n-2-\frac{1}{2}\sum_{i=1}^{q}|V(C_i)|(n-|T|-|V(C_i)|)$ edges not in $E_{G}(T,V(G)\setminus T)\cup E_{G}(T)$, and therefore,  
$$
\begin{aligned}
	\sum_{x \in T}d_{G}(x)&\geq (n-1)|T|- 2\left(n-2-\frac{1}{2}\sum_{i=1}^{q}|V(C_i)|(n-|T|-|V(C_i)|)\right)\\
	 	 &\geq (n-1)|T|-2n+4 +\sum_{i=1}^{q}2(n-|T|-|V(C_i)|)~~\mbox{(as $|V(C_i)|\geq 2$)}\\
		 &= (n-1)|T|-2n+4 +(2q-2)(n-|T|)\\
	&\geq (n-1)|T|-2n+4+2(n-|T|)\\
	&=(n-3)|T|+4.
\end{aligned}
$$
Note that $q\leq (n-|T|)/2\leq n-|T|$ (by Claim \ref{claim::2}) and $n \geq 2b^2+4b$. We have
$$
\begin{aligned}
	\delta(S,T)&=-b|T|+\sum_{x \in T}d_{G}(x)-q\\
	&\geq -b|T|+(n-3)|T|+4-(n-|T|)  \\
	&=(n-b-2)|T|-n+4  \\
	&\geq 2(n-b-2)-n+4  \\
	&=n-2b\\
	&\geq 2b^2+2b\\
	&>0,
\end{aligned}
$$
contrary to (\ref{equ::3}).

{\flushleft {\it Subcase 3.2.} $|S|\geq 1$.}

Note that $q\leq (n-|S|-|T|)/2$ by Claim \ref{claim::2}. Let $T'=V(G)\setminus(S\cup T)$. Then
$$
\begin{aligned}
e_G(T)+e_G(T,T')&\leq \sum_{x \in T}d_{G-S}(x)\\
&\leq b|T|-a|S|+q-2~~\mbox{(by \eqref{equ::3})}\\
&\leq b|T|-a|S|+\frac{n-|S|-|T|}{2}-2,
\end{aligned}$$ 
and therefore, 
$$
	\begin{aligned}
		e(G)&=e_G(S)+e_G(S,T)+e_G(S,T')+e_G(T)+e_G(T,T')+e_G(T')\\
		&\le \frac{|S|(|S|-1)}{2}+|S|\cdot|T|+|S|(n-|S|-|T|)+b|T|-a|S|+\frac{n-|S|-|T|}{2}-2\\
		&~~~~+\frac{(n-|S|-|T|)(n-|S|-|T|-1)}{2}\\
		&=\frac{(n-2)^2-((2|T|-4)n+(2a-2|T|+1)|S|-|T|^2-2b|T|+8)}{2}.
	\end{aligned}
$$
We claim that $|T|\leq 2b+2$. In fact, if $|T|\geq 2b+3$, then from the above inequality and Lemma \ref{lem::2} we can deduce  that 
$$
	\begin{aligned}
		\rho(G)&\le \sqrt{2e(G)-n+1}\\
		&\le  \sqrt{(n-2)^{2}-((2|T|-3)n+(2a-2|T|+1)|S|-|T|^2-2b|T|+7)}\\
		&\le  \sqrt{(n-2)^{2}-((2|T|-3)(|S|+|T|)+(2a-2|T|+1)|S|-|T|^2-2b|T|+7)}\\
		&=  \sqrt{(n-2)^{2}-(|T|^2-(2b+3)|T|+(2a-2)|S|+7)}\\
		&\le \sqrt{(n-2)^{2}-(|T|^2-(2b+3)|T|+7)}~~\mbox{(as $a\geq 3$)}\\
		&<n-2\\
		&<\rho(H_{n,b}),
	\end{aligned}
$$
contrary to our assumption. Hence, $|T| \le 2b+2$.

If $q=0$, then $V(G)=S\cup T$. As $e(\overline{G})\leq n-2$, we have 
$$
\begin{aligned}
	\sum_{x \in T}d_{G-S}(x)&\geq (n-1-|S|)|T|- 2(n-2)=(n-1-|S|)|T|-2n+4.
\end{aligned}
$$
Note that $b>a\geq 3$, $|S|=n-|T|$ and $n\geq 2b^2+4b$. Then
$$
\begin{aligned}
	\delta(S,T)&=a|S|-b|T|+\sum_{x \in T}d_{G-S}(x)\\
	&\geq a|S|-b|T|+(n-1-|S|)|T|-2n+4 \\
	&= (a-2)n+|T|^2-(a+b+1)|T|+4  \\
	&\geq n+|T|^2-2b|T|+4  \\
	&\geq (2b^2+4b)-b^2+4\\
	&=b^2+4b+4 \\
	&>0,
\end{aligned}
$$
contrary to (\ref{equ::3}).

If $q=1$, we have $\sum_{x \in T}d_{G-S}(x) \le b|T|-a|S|-1$. If $2\le |T|\le a$, then  $\sum_{x \in T}d_{G}(x) \le b|T|+(|T|-a)|S|-1 \le b|T|-1$, and hence there exists some vertex $x_1\in T$ such that $d_G(x_1)\leq b-1$. This implies that $G$ is a spanning subgraph of $H_{n,b}$, which is impossible. Now suppose that  $|T|>a$.  As above, we have 
$$
\begin{aligned}
	\sum_{x \in T}d_{G-S}(x)&\geq (n-1-|S|)|T|- 2(n-2)=(n-1-|S|)|T|-2n+4,
\end{aligned}
$$
and therefore,
$$
\begin{aligned}
	\delta(S,T)&=a|S|-b|T|+\sum_{x \in T}d_{G-S}(x)-q\\
	&\geq a|S|-b|T|+(n-1-|S|)|T|-2n+4-1  \\
	&\geq a|S|-b|T|+a(n-1-|S|)-2n+3   \\
	&= (a-2)n-b|T|-a+3\\
	&\geq n-b(2b+2)-b+3~~\mbox{(as $|T|\leq 2b+2$ and $b>a\geq 3$)}\\
	&\geq (2b^2+4b)-b(2b+2)-b+3  \\
	&=b+3\\
	&> 0,
\end{aligned}
$$
again contrary to (\ref{equ::3}).

If $q\geq 2$, there are at most $n-2-\frac{1}{2}\sum_{i=1}^{q}|V(C_i)|(n-|S|-|T|-|V(C_i)|)$ edges not in $E_G(T,V(G)\setminus(S\cup T))\cup E_G(T)$, and therefore,   
$$
\begin{aligned}
	\sum_{x \in T}d_{G-S}(x)&\geq (n-1-|S|)|T|- 2\left(n-2-\frac{1}{2}\sum_{i=1}^{q}|V(C_i)|(n-|S|-|T|-|V(C_i)|)\right)\\
	&\geq (n-1-|S|)|T|-2n+4+2\sum_{i=1}^{q}(n-|S|-|T|-|V(C_i)|)\\
	&= (n-1-|S|)|T|-2n+4+(2q-2)(n-|S|-|T|).
\end{aligned}
$$
Then we obtain
$$
\begin{aligned}
	\delta(S,T)&=a|S|-b|T|+\sum_{x \in T}d_{G-S}(x)-q\\
	&\geq a|S|-b|T|+(n-1-|S|)|T|-2n+4+(2q-2)(n-|S|-|T|)-q  \\
	&\geq a|S|-b|T|+(n-(|T|-1)-|S|)|T|-2n+4+q(n-|S|-|T|-1)  \\
	&= a|S|-(b-2)|T|-2n+4+(q+|T|)(n-|S|-|T|-1)   \\
	&\geq a|S|-(b-2)|T|-2n+4+3(n-|S|-|T|-1)   \\
	&\geq (a-3)|S|+n-(b+1)|T|+1 ~~\mbox{(as $|T|\leq 2b+2$ and $a\geq 3$)}\\
	&\geq n-(b+1)(2b+2)+1\\
	&\geq (2b^2+4b)-(b+1)(2b+2)+1\\
	&=-1,
\end{aligned}
$$
which is contrary to (\ref{equ::3}).

Therefore, we deduce a contradiction in all situations, and the result follows.  \qed

{\flushleft \it Proof of Theorem \ref{thm::8}.}  Since $n \geq 3b(b+1)/a+3b+7>b+2$, by Lemma \ref{lem::4}, the graph $H_{n,b}$ cannot have all fractional $[a, b]$-factors. As in the proof of Theorem \ref{thm::7}, we see that $\rho(G) \geq \rho(H_{n,b}) > n-2$ and $G$ is connected.  Suppose to the contrary that $G$ does not have all fractional $[a, b]$-factors and $G\ncong H_{n,b}$.  By Theorem \ref{thm::6}, there exists some subset $S\subseteq V(G)$ such that
\begin{equation}\label{equ::4}
	\theta(S,T)=a|S|-b|T|+\sum_{x \in T}d_{G-S}(x) \le -1,
\end{equation}
where $T=\{x:x\in V(G)\backslash S, d_{G-S}(x)< b\}$. We claim that $T\neq \emptyset$, since otherwise we have $\theta(S,T)=a|S|-b|T|+\sum_{x \in T}d_{G-S}(x)=a|S|\ge 0$, contrary to \eqref{equ::4}. Thus it remains to consider the following two situations.

{\flushleft {\it Case 1.} $|T|=1$.}

In this situation, suppose $T=\{x_{0}\}$.  By \eqref{equ::4}, we obtain $d_{G-S}(x_{0}) \le b-a|S|-1$, and so $d_G(x_{0}) \le b-a|S|+|S|-1=b+(1-a)|S|-1 \le b-1$. This implies that $G$ is a spanning subgraph of $H_{n,b}$, which is impossible.

{\flushleft {\it Case 2.} $|T|\geq 2$.}

In this situation,  we claim that $|T|\leq 2b+2$. By contradiction,  suppose that $|T|\geq 2b+3$. Let $T'=V(G)\setminus(S\cup T)$. Again by \eqref{equ::4}, we have
$$e_G(T)+e_G(T,T')\leq \sum_{x \in T}d_{G-S}(x) \le b|T|-a|S|-1,$$ 
and therefore,
$$
	\begin{aligned}
		e_G(G)&=e_G(S)+e_G(S,T)+e_G(S,T')+e_G(T)+e_G(T,T')+e_G(T')\\
		&\le \frac{|S|(|S|-1)}{2}+|S||T|+|S|(n-|S|-|T|)+(b|T|-a|S|-1)\\
		&~~~~+\frac{(n-|S|-|T|)(n-|S|-|T|-1)}{2}\\
		&=\frac{(n-2)^{2}-((2|T|-3)n-|T|^{2}-(2|S|+2b+1)|T|+2a|S|+6)}{2}.
	\end{aligned}
$$
Then from Lemma \ref{lem::2} we deduce that
$$
	\begin{aligned}
		\rho(G)&\le \sqrt{2e(G)-n+1}\\
		&\le \sqrt{(n-2)^{2}-((2|T|-2)n-|T|^2-(2|S|+2b+1)|T|+2a|S|+5)}\\
		&\le \sqrt{(n-2)^{2}-((2|T|-2)(|S|+|T|)-|T|^2-(2|S|+2b+1)|T|+2a|S|+5)}\\
		&=\sqrt{(n-2)^{2}-(|T|^2-(2b+3)|T|+(2a-2)|S|+5)}\\
		&< n-2~~\mbox{(as $|T|\geq 2b+3$ and $a\geq 1$)}\\
		&< \rho(H_{n,b}),
	\end{aligned}
$$
contrary to our assumption. Hence, $|T| \le 2b+2$. 

{\flushleft {\it Subcase 2.1.} $|S|>|T|$.}

In this situation, if $a|S|\ge b|T|$, then $\theta(S,T)=a|S|-b|T|+\sum_{x \in T}d_{G-S}(x)\geq 0$, contrary to \eqref{equ::4}. Hence, $a|S|<b|T|$. Since $|T|\le 2b+2$, we have $|S|<b|T|/a\le (2b^2+2b)/a$, and so $|S|\leq \lceil(2b^2+2b)/a\rceil-1$. Also, we assert that there exists some vertex $x_1\in T$ such that $d_{G-S}(x_{1})\leq b-2$, since otherwise we can deduce from \eqref{equ::4} that $a|S|-|T| \le -1$, which is impossible because $|S|>|T|$ and $a\ge 1$. As $|T|\geq 2$, we take  $x_2\in T$ with $x_1\neq x_2$. Clearly, $d_{G-S}(x_2)\leq b-1$ by the definition of $T$. Then 
$$
	\begin{aligned}
	|(N_G(x_1)\setminus\{x_2\})\cup (N_G(x_2)\setminus\{x_1\})|&\leq |S|+|(N_{G-S}(x_1)\setminus\{x_2\})\cup (N_{G-S}(x_2)\setminus\{x_1\})|\\
	&<\left(\left\lceil\frac{2b^2+2b}{a}\right\rceil-1\right)+(b-1)+(b-2)\\
	&=\left\lceil\frac{2b^2+2b}{a}\right\rceil+2b-4,
	\end{aligned}
	$$
which implies that $G$ is a spanning subgraph of $G_1:=K_{\lceil(2b^2+2b)/a\rceil+2b-4} \nabla (K_{2} \cup K_{n-\lceil(2b^2+2b)/a\rceil-2b+2})$.  Therefore, by Lemma \ref{lem::3}, we obtain $\rho(G)\leq \rho(G_1)< n-2 < \rho(H_{n,b})$, contrary to our assumption.

{\flushleft {\it Subcase 2.2.} $|S|\leq |T|$.}

Let  $x_3, x_4\in T$ with $x_3\neq x_4$. Note that  $|S|\leq |T|\leq 2b+2$. Then we have
$$
	\begin{aligned}
	|(N_G(x_3)\setminus\{x_4\})\cup (N_G(x_4)\setminus\{x_3\})|&\leq |S|+|(N_{G-S}(x_3)\setminus\{x_4\})\cup (N_{G-S}(x_4)\setminus\{x_3\})|\\
	&\leq (2b+2)+2(b-1)\\
	&=4b,
	\end{aligned}$$
and hence $G$ is a spanning subgraph of $G_2:=K_{4b} \nabla (K_{2} \cup K_{n-4b-2})$. Combining this with Lemma \ref{lem::3}, we obtain
$\rho(G)\leq \rho(G_2)< n-2 < \rho(H_{n,b})$, which is also impossible.

This completes the proof.  \qed

\section*{Acknowledgement}

X. Huang is supported by the National Natural Science Foundation of China (Grant No. 11901540).

\end{document}